\documentclass[10pt,a4paper,reqno]{amsart}
\usepackage{amsmath,amssymb,amsthm,graphicx,mathrsfs}

\usepackage[usenames,dvipsnames]{color}
\usepackage{dsfont}   %Added by Markus
\definecolor{darkred}{rgb}{0.4,0.1,0.1}
\definecolor{darkblue}{rgb}{0.1,0.1,0.4}
\definecolor{darkgrey}{rgb}{0.5,0.5,0.5}

\usepackage{lmodern}

\numberwithin{equation}{section}
\theoremstyle{plain}% default
\newtheorem{theorem}{Theorem}[section]
\newtheorem{lemma}[theorem]{Lemma}
\newtheorem{proposition}[theorem]{Proposition}
\newtheorem{corollary}[theorem]{Corollary}

\theoremstyle{definition}
\newtheorem{definition}[theorem]{Definition}
\newtheorem{remark}[theorem]{Remark}

\newcommand{\drm}{{\mathrm d}}

\newcommand\cB{\mathcal B}

\newcommand\cG{\mathcal G}
\newcommand\cH{\mathcal H}

\newcommand\cL{\mathcal L}

\newcommand\NN{\mathbb N}
\newcommand\RR{\mathbb R}

\newcommand\fra{\mathfrak a}

\newcommand\eps{\varepsilon}
\renewcommand{\phi}{\varphi}

\newcommand\ov{\overline}
\newcommand\wt{\widetilde}

\newcommand{\defeq}{\mathrel{\mathop:}=}

\newcommand{\rmO}{{\rm O}}

\DeclareMathOperator\dom{dom}
\DeclareMathOperator\ran{ran}

\DeclareMathOperator{\essinf}{ess\,inf}

\newcommand\void[1]{}

\newcounter{counter_a}
\newenvironment{myenum}{\begin{list}{{\rm(\roman{counter_a})}}%
{\usecounter{counter_a}
\setlength{\itemsep}{0.5ex}\setlength{\topsep}{0.7ex}
\setlength{\leftmargin}{7ex}\setlength{\labelwidth}{7ex}}}{\end{list}}
\newcounter{counter_b}
\newenvironment{myenuma}{\begin{list}{{\rm(\alph{counter_b})}}%
{\usecounter{counter_b}
\setlength{\itemsep}{0.5ex}\setlength{\topsep}{0.7ex}
\setlength{\leftmargin}{7ex}\setlength{\labelwidth}{7ex}}}{\end{list}}

\def\frt{{\mathfrak t}}

      \def\dC{{\mathbb C}}

   \def\dN{{\mathbb N}}   
      \def\dR{{\mathbb R}}
\def\dS{{\mathbb S}}      
      
   \def\dZ{{\mathbb Z}}

   \def\cB{{\mathcal B}}   
      
\def\cG{{\mathcal G}}   \def\cH{{\mathcal H}}   
      \def\cL{{\mathcal L}}
      \def\cO{{\mathcal O}}

\title{Quasi boundary triples and semibounded self-adjoint extensions}
\author[J.~Behrndt]{Jussi Behrndt}
\address{Technische Universit\"{a}t Graz,
Institut f\"{u}r Numerische Mathematik,
Steyrergasse 30,
8010 Graz, Austria}
\email{behrndt@tugraz.at}
\author[M.~Langer]{Matthias Langer}

\address{Department of Mathematics and Statistics,
University of Strathclyde,
26 Richmond Street, Glasgow G1 1XH, United Kingdom}
\email{m.langer@strath.ac.uk}

\author[V.~Lotoreichik]{Vladimir Lotoreichik}

\address{Department of Theoretical Physics, Nuclear Physics Institute ASCR,
250 68 \v{R}e\v{z} near Prague, Czech Republic}
\email{lotoreichik@ujf.cas.cz}

\author[J.~Rohleder]{Jonathan Rohleder}

\address{Stockholms universitet,
Matematik, 106 91 Stockholm, Sweden}
\email{jonathan.rohleder@math.su.se}

\begin{document}

\begin{abstract}
In this note semibounded self-adjoint extensions of symmetric operators are
investigated with the help of the abstract notion of quasi boundary triples and
their Weyl functions.  The main purpose is to provide new sufficient conditions on
the parameters in the boundary space to induce self-adjoint realizations, and to relate the
decay of the Weyl function to estimates on the lower bound of the spectrum.
The abstract results are illustrated with uniformly elliptic second-order PDEs on
domains with non-compact boundaries.
\end{abstract}


\maketitle

% *******************************************************************
\section{Introduction}
% *******************************************************************

\noindent
Boundary triple techniques are nowadays a widely used abstract tool in
extension theory and spectral analysis of symmetric and self-adjoint operators.
These methods are inspired by, and can be viewed as, abstract counterparts of
trace maps for ordinary or partial differential operators.
The Titchmarsh--Weyl $m$-function in singular Sturm--Liouville theory and
the Dirichlet-to-Neumann map in the analysis of elliptic differential operators
correspond to the Weyl function associated with such a boundary triple.
We refer the reader to \cite{BGP08,DM91,GG91,S12} for ordinary boundary triples,
typical applications and further references, to \cite[Chapter 13]{G09} and \cite{BGW09,GM11,G68,M10,P08,P13}
for extension theory of partial differential operators, and to
\cite{BL07,BL12,BLL13IEOT,BM14} for the more general notion of
quasi boundary triples and their use in the spectral analysis of
partial differential operators.

The usual starting point is a densely defined, closed, symmetric operator $S$ with equal
deficiency indices $n_\pm(S)$ in a Hilbert space $\cH$ and its adjoint $S^*$.
An \emph{ordinary boundary triple} $\{\cG,\Gamma_0,\Gamma_1\}$ for $S^*$
consists of a Hilbert space $\cG$ with $\dim\cG=n_\pm(S)$
and two boundary mappings $\Gamma_0,\Gamma_1:\dom S^*\rightarrow\cG$ that
satisfy an abstract Lagrange or Green identity
\begin{equation}\label{greenintro}
  (S^*f,g)_\cH-(f,S^*g)_\cH=(\Gamma_1 f,\Gamma_0 g)_\cG-(\Gamma_0 f,\Gamma_1 g)_\cG, \quad
  f,g\in\dom S^*,
\end{equation}
and a maximality condition.
With the help of a boundary triple the self-adjoint
extensions of $S$ in $\cH$ can be parameterized in an efficient
way via abstract boundary conditions in the boundary space $\cG$.
More precisely, the restriction
\begin{equation}\label{abintro}
  A_{[B]} f \defeq S^*f,\qquad
  \dom A_{[B]}=\bigl\{f\in\dom S^*:\Gamma_0 f=B \Gamma_1 f\bigr\},
\end{equation}
of $S^*$ is self-adjoint in $\cH$ if and only if $B$ is a self-adjoint
operator or relation in the boundary space $\cG$.  The spectral properties
of the self-adjoint extensions $A_{[B]}$ can be investigated with the help
of the \emph{Weyl function} $M$ associated with the
boundary triple $\{\cG,\Gamma_0,\Gamma_1\}$.
The values $M(\lambda)$ of the Weyl function are linear
operators in $\cG$ defined by
\begin{equation}\label{mintro}
  M(\lambda):\cG\rightarrow\cG,\quad
  \Gamma_0 f \mapsto \Gamma_1 f, \qquad
  f\in\ker(S^*-\lambda),
\end{equation}
where $\lambda\in\dC$ does not belong to the spectrum of the
self-adjoint extension $A_0 \defeq S^*\upharpoonright\ker\Gamma_0$.
It turns out that $M$ belongs to the class of operator-valued Nevanlinna or
Riesz--Herglotz functions, and, very roughly speaking, the spectral properties
of a self-adjoint extension $A_{[B]}$ of $S$ in \eqref{abintro} are encoded
in the singularities of the function $\lambda\mapsto (B^{-1}-M(\lambda))^{-1}$.

For many purposes and applications the notion of boundary triples and their
Weyl functions is an efficient and most suitable tool, in particular,
for ordinary differential operators and all other extension problems
where the deficiency indices of the underlying symmetric operator are finite.
However, if one tries to apply the boundary triple method to elliptic PDEs
on bounded or unbounded domains with the usual Dirichlet and Neumann trace
as boundary maps and the Dirichlet-to-Neumann map as Weyl function, one
gets into very serious trouble since Green's second identity does not extend to all functions in the maximal domain.
There are various ways to overcome this technical difficulty, see \cite{BL07,BL12,BGW09,GM11,M10,P08,P13,R07} and the classical contributions \cite{G68,V52} for more details.
One possible solution is the concept of \emph{quasi boundary triples},
which is a slight generalization of the notion of boundary triples and
which was proposed in \cite{BL07} and further developed
and applied in, e.g.\ \cite{BL12,BLL13IEOT,BM14}.
The key idea is to define the boundary maps only on a suitable core of the
adjoint operator $S^*$ and to require \eqref{greenintro} to hold only for elements in this core.
The notion of the Weyl function in \eqref{mintro} remains almost the same:
instead of all defect elements $f\in\ker(S^*-\lambda)$, only those
belonging to the core are allowed; see Section~\ref{sec2} below.
However, it is important to note that for quasi boundary triples some of the
striking properties of ordinary boundary triples fail, e.g.\ self-adjointness
of $B$ does not imply, in general, self-adjointness of the extension $A_{[B]}$ in \eqref{abintro}.
Therefore it is desirable to find sufficient conditions for the boundary parameters $B$
to induce self-adjoint extensions $A_{[B]}$ via \eqref{abintro}.
There are some useful sufficient conditions in the literature, most of which rely
on compactness properties of the Weyl function; see, e.g.\ \cite[Theorems 6.20 and 6.21]{BL12} or \cite{BL07,BLL13IEOT,BM14}.

One of the main aims of the present paper
is to provide new sufficient conditions for the
boundary parameter $B$ to induce a self-adjoint extension $A_{[B]}$
via \eqref{abintro} in the framework of quasi boundary triples.
In Theorems~\ref{mainthm} and \ref{mainthm2} and Corollaries~\ref{maincor}
and \ref{maincor2} we drop the above mentioned compactness assumptions and replace
them by a set of abstract conditions.  In the special but important case
when $A_0$ is bounded from below and $M(\lambda)\to0$
in the operator norm as $\lambda \rightarrow - \infty$, these conditions simplify further;
see Theorem~\ref{thm:lowerBound}.
We emphasize that in the present setting also unbounded
self-adjoint operators $B$ are allowed in~\eqref{abintro}.

Our second main objective is to relate decay properties
of the Weyl function associated with an ordinary or quasi boundary triple
to the lower bounds of the spectra of the self-adjoint extensions $A_{[B]}$.
More precisely, since $M$ is a Nevanlinna function, it behaves similarly to
the resolvent of the self-adjoint operator $A_0$.  If $A_0$ is bounded
from below, the decay of the Weyl function $\lambda\mapsto M(\lambda)$
for $\lambda\rightarrow -\infty$ may be like % estimated by
\begin{equation}\label{decayintro}
  \bigl\|\ov{M(\lambda)}\bigr\| = \rmO\biggl(\frac{1}{(\mu-\lambda)^\alpha}\biggr)
  \qquad \text{as}\;\;\lambda\to-\infty,
\end{equation}
for some $\alpha\in(0,1]$ and a certain $\mu \leq \min \sigma(A_0)$.
This leads to an estimate for the lower bound of the self-adjoint
extensions $A_{[B]}$ when the norm of the parameter $B$ tends to $\infty$;
see Theorem~\ref{thm:lowerBound} and Corollary~\ref{cor:Asymptotics}.

Our general considerations and results in Section~\ref{sec2} are partly
inspired by possible applications to elliptic PDEs on unbounded
domains with non-compact boundaries.  In Section~\ref{sec:PDE} we illustrate
our methods with uniformly elliptic second-order
differential expressions with smooth variable coefficients.
The boundary maps $\Gamma_0$ and $\Gamma_1$ are chosen to be
the Neumann and Dirichlet trace, respectively, defined on $H^2(\Omega)$,
which is a core for the maximal operator.  In that case the Weyl function $M$
is the Neumann-to-Dirichlet map, and it is shown in Proposition~\ref{prop:Weyl_elliptic}
that the norms of the closures $\overline{M(\lambda)}$ satisfy \eqref{decayintro}
with $\alpha=\frac{1}{2}$.
As a consequence, the abstract results in Section~\ref{sec2} yield self-adjointness
and an estimate for the lower bounds of the spectra of the
self-adjoint realizations $A_{[B]}$ in terms of the
boundary parameter $B$ in Theorem~\ref{thm:mainApplication}.
Here the parameter $B$ in the local or non-local
Robin boundary condition may also be an unbounded operator.
We mention that in certain cases similar estimates can also be obtained via
standard techniques involving quadratic forms; see Remark~\ref{formremark}.
Finally, we refer the reader to~\cite{GM08,GM09,G11,G12,G12-1}
for a small selection of other recent contributions on spectral properties
of elliptic differential operators and especially to~\cite{AGW14, G08, LR12, R15, VV06}
and the monographs~\cite{G96, Vo} for elliptic operators on domains
with non-compact boundaries.  For further recent contributions to the literature on asymptotics
of lower bounds and more explicit spectral asymptotics for elliptic
differential operators with Robin boundary conditions we refer the
reader to~\cite{FK,HK,HP15,LP08,PP} and their references.

\subsection*{Acknowledgements}

Jussi Behrndt, Vladimir Lotoreichik and Jonathan Rohleder
acknowledge financial support
by the Austrian Science Fund (FWF): Project P~25162-N26.
V.\,L. also acknowledges financial support
by the Czech Science Foundation: Project 14-06818S.
The authors wish to thank Gerd Grubb for various fruitful discussions and
helpful remarks, and, in particular, for pointing out the article~\cite{A62},
which led to the proof of Proposition~\ref{prop:Weyl_elliptic}\,(iv) in its present form.

% *******************************************************************
\section{Quasi boundary triples, Weyl functions and \\ self-adjoint extensions}\label{sec2}
% *******************************************************************

\noindent
Throughout this section we assume that $S$ is a densely defined, closed,
symmetric operator in a Hilbert space $\cH$.  We start by recalling the notion
of quasi boundary triples, which was introduced in \cite{BL07} as a
generalization of the concepts of ordinary and generalized boundary triples;
for the latter see, e.g.\ \cite{DM91,DM95}.

In the following we denote all appearing inner products by $(\cdot\,,\cdot)$;
the respective Hilbert space will be clear from the context.

\begin{definition}
Let $T\subset S^*$ be a linear operator in $\cH$ such that $\overline T=S^*$.
A triple $\{\cG,\Gamma_0,\Gamma_1\}$ is called a \emph{quasi boundary triple}
for $T\subset S^*$ if\, $\cG$ is a Hilbert space and $\Gamma_0,\Gamma_1:\dom T\rightarrow\cG$
are linear mappings such that
\begin{myenum}
  \item the abstract Green identity
    \begin{equation}\label{green}
      (T f,g) - (f,Tg)
      = (\Gamma_1 f,\Gamma_0 g)
      - (\Gamma_0 f,\Gamma_1 g)
    \end{equation}
    holds for all $f,g\in\dom T$;
  \item the map $\Gamma \defeq (\Gamma_0,\Gamma_1)^\top : \dom T \rightarrow
    \mathcal G \times \mathcal G$ has dense range;
  \item $A_0 \defeq T\upharpoonright \ker\Gamma_0$ is a self-adjoint operator in $\cH$.
\end{myenum}
\end{definition}

\noindent
We recall from \cite{BL07,BL12} that a quasi boundary triple exists if and only
if $S$ admits self-adjoint extensions in $\cH$, that is,
the deficiency indices of $S$ are equal.
Moreover, if $\{\cG,\Gamma_0,\Gamma_1\}$ is a quasi boundary triple for $T\subset S^*$,
then one has $T=S^*$ if and only if $\ran \Gamma = \cG \times \cG$, in which case
$\Gamma= (\Gamma_0,\Gamma_1)^\top : \dom S^* \to \cG \times \cG$
is onto and continuous with respect to the graph norm of $S^*$, the abstract
Green identity holds for all $f,g\in\dom S^*$, and the
restriction $A_0=S^*\upharpoonright \ker\Gamma_0$ is automatically self-adjoint.
In this situation the notion of quasi boundary triples
coincides with the notion of ordinary boundary triples.  In particular, this is
the case when the deficiency indices of $S$ are finite (and equal). For later use let
us also introduce the notation $A_1\defeq T\upharpoonright \ker\Gamma_1$. In contrast to the
case of an ordinary boundary triple, this extension of~$S$ is not necessarily self-adjoint.

With each quasi boundary triple $\{\cG,\Gamma_0,\Gamma_1\}$ one associates
a so-called $\gamma$-field and a Weyl function.
Before we recall their definitions, note that for each $\lambda\in\rho(A_0)$ one
has the direct sum decomposition
\begin{equation*}
  \dom T = \dom A_0\,\dot +\,\ker(T-\lambda)
  = \ker\Gamma_0\,\dot +\,\ker(T-\lambda).
\end{equation*}
Thus the restriction of the boundary map $\Gamma_0$ to $\ker(T-\lambda)$
is injective, and its range coincides with $\ran\Gamma_0$.
The definitions of the $\gamma$-field and the Weyl function are now formally
the same as for ordinary and generalized boundary triples.

\begin{definition}
The \emph{$\gamma$-field} $\gamma$ and the \emph{Weyl function} $M$ corresponding to the
quasi boundary triple $\{\mathcal G,\Gamma_0,\Gamma_1\}$ are defined by
\begin{equation*}
  \lambda \mapsto \gamma(\lambda)\defeq\bigl(\Gamma_0\upharpoonright\ker(T-\lambda)\bigr)^{-1},
  \qquad \lambda \in \rho(A_{0}),
\end{equation*}
and
\begin{equation*}
  \lambda \mapsto M(\lambda) \defeq \Gamma_1 \gamma(\lambda),
  \qquad \lambda \in \rho(A_{0}),
\end{equation*}
respectively.
\end{definition}

Observe that $\gamma(\lambda)$ is a mapping from $\ran\Gamma_0\subset\cG$
onto $\ker(T-\lambda)\subset\cH$ and that the values $M(\lambda)$ of the Weyl
function are operators in $\cG$ mapping $\ran\Gamma_0$ into $\ran\Gamma_1$.
Note that $\ran\Gamma_0$ and $\ran\Gamma_1$ are both dense subspaces of $\cG$;
this is a consequence of the density of the range of $\Gamma=(\Gamma_0,\Gamma_1)^\top$.
Various useful and important properties of the $\gamma$-field and the Weyl function
can be found in \cite[Proposition~2.6]{BL07} or \cite[Propositions~6.13 and 6.14]{BL12}.
For later purposes we recall that the adjoint $\gamma(\lambda)^*$
is a bounded, everywhere defined operator from $\cH$ to $\cG$, which satisfies
\begin{equation}\label{gammastern}
  \gamma(\lambda)^*=\Gamma_1(A_0-\overline\lambda)^{-1},\qquad \lambda\in\rho(A_0).
\end{equation}
Furthermore, the values of the Weyl function have the property
$M(\lambda)\subset M(\overline\lambda)^*$, $\lambda\in\rho(A_0)$, and,
in particular, the operators $M(\lambda)$ are closable.
We point out that the operators $M(\lambda)$ and their closures $\overline{M(\lambda)}$
are in general not bounded.
However, if $M(\lambda_0)$ is bounded for one $\lambda_0\in\rho(A_0)$, then
$M(\lambda)$ is bounded for all $\lambda\in\rho(A_0)$;
see \cite[Proposition~3.3\,(viii)]{BLL13IEOT}.
The next lemma, which contains further properties of the Weyl function, is used later.

\begin{lemma}\label{le:Mnonneg}
Let $\{\cG,\Gamma_0,\Gamma_1\}$ be a quasi boundary triple for $T\subset S^*$
with corresponding Weyl function $M$.
\begin{myenum}
\item
For every $\varphi\in(\ran\Gamma_0)\setminus\{0\}$ the function
\[
  \lambda \mapsto \bigl(M(\lambda)\varphi,\varphi\bigr)
\]
is strictly increasing on each interval in $\rho(A_0)\cap\RR$.
\item
If $A_0$ is bounded from below and
\begin{equation}\label{conv0}
  (M(\lambda)\varphi,\varphi)\to0 \qquad\text{as}\;\;\lambda\to-\infty
\end{equation}
for all $\varphi\in\ran\Gamma_0$,
then
\begin{equation}\label{Mpositive}
  (M(\lambda)\varphi,\varphi)>0, \qquad \varphi\in(\ran\Gamma_0)\setminus\{0\},\;
  \lambda<\min\sigma(A_0).
\end{equation}
\end{myenum}
\end{lemma}

\begin{proof}
(i)
For $\varphi\in(\ran\Gamma_0)\setminus\{0\}$ and $\lambda\in\rho(A_0)\cap\RR$
we obtain from \cite[Proposition~3.3\,(vii)]{BLL13IEOT} that
\[
  \frac{\drm}{\drm\lambda}\bigl(M(\lambda)\varphi,\varphi\bigr)
  = \bigl(\gamma(\lambda)^*\gamma(\lambda)\varphi,\varphi\bigr)
  = \|\gamma(\lambda)\varphi\|^2 > 0
\]
where the last inequality is true since $\gamma(\lambda)$ is injective.

(ii)
Relation \eqref{Mpositive} follows directly from \eqref{conv0} and (i).
\end{proof}

In contrast to ordinary boundary triples there is no one-to-one
correspondence between self-adjoint relations $\Theta$ or $B$ in $\cG$ and self-adjoint
extensions of $S$ in $\cH$ of the
form $A_\Theta=S^*\upharpoonright\ker(\Gamma_1-\Theta\Gamma_0)$
or $A_{[B]}=S^*\upharpoonright\ker(B\Gamma_1-\Gamma_0)$, respectively.
However, various sufficient conditions for self-adjointness
in terms of the parameters $\Theta$ or $B$
were obtained in,
e.g.\ \cite[Theorems~6.20 and 6.21]{BL12} and \cite[Theorem~3.11]{BLL13IEOT},
and, in connection with PDEs on domains with compact boundaries,
also in \cite[Proposition~4.3 and Theorem~4.8]{BL07}
and \cite[Theorem~4.5]{BLL13IEOT}.
In the next theorem we provide a new very useful
sufficient condition, which is formulated for the parameter $B=\Theta^{-1}$.
In contrast to earlier results no compactness assumption on the values of the
Weyl function is imposed.  In particular, this allows us to apply the
abstract results to elliptic PDEs on domains with non-compact boundaries;
see Section~\ref{sec:PDE}. We remark that in the application the
conditions on~$B$ simplify substantially.

% -------------------------------------------------------------------
\begin{theorem}\label{mainthm}
Let $\{\cG,\Gamma_0,\Gamma_1\}$ be a quasi boundary triple for $T\subset S^*$
with corresponding $\gamma$-field $\gamma$ and Weyl function $M$.
Let $B$ be a linear operator in $\cG$ and assume that there
exist $\lambda_\pm\in\dC^\pm$ such that the following conditions are satisfied:
\begin{myenum}
 \item $B$ is symmetric;
 \item $1\in\rho(B\overline{M(\lambda_\pm)})$;
 \item $B\bigl(\ran\ov{M(\lambda_\pm)}\cap\dom B\bigr)\subset\ran\Gamma_0$;
 \item $\ran\Gamma_1\subset\dom B$;
 \item $B(\ran\Gamma_1)\subset\ran\Gamma_0$ \;\; or \;\; $A_1$ is self-adjoint.
\end{myenum}
Then the operator
\begin{equation}\label{ab}
  A_{[B]}f=Tf, \qquad
  \dom A_{[B]}=\bigl\{f\in\dom T:\Gamma_0 f= B\Gamma_1 f\bigr\},
\end{equation}
is a self-adjoint extension of $S$, and
\begin{equation}\label{resformula}
  (A_{[B]}-\lambda)^{-1}=(A_0-\lambda)^{-1}
  +\gamma(\lambda)\bigl(I-BM(\lambda)\bigr)^{-1}B\gamma(\overline\lambda)^*
\end{equation}
holds for all $\lambda\in\rho(A_{[B]})\cap\rho(A_0)$.
\end{theorem}

Note that if $\{\cG,\Gamma_0,\Gamma_1\}$ is a generalized boundary triple,
i.e.\ if $\ran\Gamma_0=\cG$, then (iii) and (v) are automatically satisfied.

Before we prove Theorem~\ref{mainthm}, we state a corollary for bounded $B$,
which follows immediately from Theorem~\ref{mainthm}.

\begin{corollary}\label{maincor}
Let $\{\cG,\Gamma_0,\Gamma_1\}$ be a quasi boundary triple for $T\subset S^*$
with corresponding $\gamma$-field $\gamma$ and Weyl function $M$.
Let $B$ be a bounded self-adjoint operator in $\cG$ and assume that there
exist $\lambda_\pm\in\dC^\pm$ such that the following conditions are satisfied:
\begin{myenum}
\item
$1\in\rho(B\overline{M(\lambda_\pm)})$;
\item
$B(\ran\ov{M(\lambda_\pm)})\subset\ran\Gamma_0$;
\item
$B(\ran\Gamma_1)\subset\ran\Gamma_0$ \;\; or \;\; $A_1$ is self-adjoint.
\end{myenum}
Then the operator $A_{[B]}$ in \eqref{ab} is a self-adjoint extension of $S$,
and the resolvent formula \eqref{resformula} holds for
all $\lambda\in\rho(A_{[B]})\cap\rho(A_0)$.
\end{corollary}

\begin{proof}[Proof of Theorem~\ref{mainthm}]
The proof of Theorem~\ref{mainthm} consists of several steps.
In the first four steps we assume that the first condition in (v) is satisfied.

\medskip
\noindent
\textit{Step 1.}
First we show that $A_{[B]}$ is symmetric, which is essentially a simple
consequence of the abstract Green identity \eqref{green}.  In fact, by assumption (iv)
for $f,g\in\dom A_{[B]}$ we have $\Gamma_1f,\Gamma_1g\in\dom B$,
\begin{equation*}
  B\Gamma_1 f=\Gamma_0 f,\quad\text{and}\quad B\Gamma_1 g=\Gamma_0 g,
\end{equation*}
which implies that
\begin{align*}
  (A_{[B]}f,g)-(f,A_{[B]}g)&=(Tf,g)-(f,Tg) =(\Gamma_1 f,\Gamma_0 g)-(\Gamma_0 f,\Gamma_1 g) \\
  &=(\Gamma_1 f,B\Gamma_1 g)-(B\Gamma_1 f,\Gamma_1 g) = 0,
\end{align*}
where assumption (i) on the symmetry of the operator $B$ was used
in the last step.  This shows that $A_{[B]}$ is a symmetric operator in $\cH$.

\medskip
\noindent
\textit{Step 2.}
In this step we show the inclusions
\begin{equation}\label{inclusion}
  \ran\bigl(B\gamma(\overline\lambda_\pm)^*\bigr) \subset \ran\bigl( I-BM(\lambda_\pm)\bigr).
\end{equation}
We consider only $\lambda_+\in\dC^+$; the proof for $\lambda_-\in\dC^-$ is the same.
Note first that it follows from \eqref{gammastern} and condition (iv) that the
product $B\gamma(\overline\lambda_\pm)^*$ is everywhere defined.
Let $g\in\ran(B\gamma(\ov\lambda_+)^*)$. Then there exists an $f\in\cH$
such that $g=B\gamma(\ov\lambda_+)^*f$.
By \eqref{gammastern} we have
$\gamma(\ov\lambda_+)^* f=\Gamma_1(A_0-\lambda_+)^{-1}f\in\ran\Gamma_1$,
and hence assumption (v) implies that
\begin{equation}\label{gutgut}
  B\gamma(\overline\lambda_+)^* f\in\ran\Gamma_0.
\end{equation}
We set
\begin{equation}\label{varphi}
 \varphi\defeq\bigl(I-B\overline{M(\lambda_+)}\,\bigr)^{-1}B\gamma(\overline\lambda_+)^* f,
\end{equation}
which is well defined by assumption (ii).
We can rewrite \eqref{varphi} in the form
\begin{equation}\label{varphi2}
  \varphi=B\overline{M(\lambda_+)}\varphi + B\gamma(\overline\lambda_+)^* f.
\end{equation}
Since $\ov{M(\lambda_+)}\varphi \in \ran \ov{M(\lambda_+)}\cap\dom B$, assumption (iii)
and the relations \eqref{gutgut} and \eqref{varphi2} imply
that $\varphi\in\ran\Gamma_0=\dom M(\lambda_+)$.
Together with \eqref{varphi2} this yields
\begin{equation*}
 \bigl(I-BM(\lambda_+)\bigr)\varphi=B\gamma(\overline\lambda_+)^* f = g,
\end{equation*}
and hence $g\in\ran(I-BM(\lambda_+))$, i.e.\
the inclusion \eqref{inclusion} is shown for $\lambda_+\in\dC^+$.

\medskip
\noindent
\textit{Step 3.}
We claim that $\ran (A_{[B]}-\lambda_\pm)=\cH$ holds.
Again we show the assertion only for $\lambda_+\in\dC^+$;
the arguments for $\lambda_-\in\dC^-$ are the same.
Let $f\in\cH$ and consider the element
\begin{equation}\label{h}
  h \defeq (A_0-\lambda_+)^{-1}f+\gamma(\lambda_+)\bigl(I-BM(\lambda_+)\bigr)^{-1}
  B\gamma(\overline \lambda_+)^* f.
\end{equation}
Note that by assumption (ii) the inverse $(I-BM(\lambda_+))^{-1}$ exists.
It maps into $\dom M(\lambda_+)=\ran\Gamma_0$, so the product
with $\gamma(\lambda_+)$ is well defined. Observe also that the product
of $(I-BM(\lambda_+))^{-1}$ and $B\gamma(\overline \lambda_+)^*$ is well defined
by \eqref{inclusion}.
We now show that $h\in\dom A_{[B]}$.  Clearly, $h\in\dom T$ since
\begin{equation*}
  (A_0-\lambda_+)^{-1}f\in\dom A_0\subset\dom T
\end{equation*}
 and
\begin{equation*}
  \ran\gamma(\lambda_+)=\ker(T-\lambda_+)\subset\dom T.
\end{equation*}
Furthermore, using \eqref{gammastern} and the definition of $M(\lambda_+)$ we have
\begin{align*}
  B\Gamma_1 h
  &= B\Gamma_1 (A_0-\lambda_+)^{-1}f
    +B\Gamma_1 \gamma(\lambda_+)\bigl(I-BM(\lambda_+)\bigr)^{-1}B\gamma(\ov\lambda_+)^* f \\
  &= B\gamma(\overline \lambda_+)^* f
    + BM(\lambda_+)\bigl(I-BM(\lambda_+)\bigr)^{-1}B\gamma(\ov\lambda_+)^* f \\
  & = \big[(I - B M (\lambda_+)) + B M (\lambda_+) \big] \bigl(I-BM(\lambda_+)\bigr)^{-1}B\gamma(\ov\lambda_+)^* f \\
  &= \bigl(I-BM(\lambda_+)\bigr)^{-1}B\gamma(\ov\lambda_+)^* f;
\end{align*}
the relation $\dom A_0=\ker\Gamma_0$ and the definition of $\gamma(\lambda_+)$ yield
\begin{align*}
  \Gamma_0 h
  &= \Gamma_0 (A_0-\lambda_+)^{-1}f
    +\Gamma_0 \gamma(\lambda_+)\bigl(I-BM(\lambda_+)\bigr)^{-1}B\gamma(\ov\lambda_+)^* f \\
  &=\bigl(I-BM(\lambda_+)\bigr)^{-1}B\gamma(\ov\lambda_+)^*f.
\end{align*}
Hence the element $h$ in \eqref{h} satisfies the boundary
condition $\Gamma_0h=B\Gamma_1 h$.  This shows that $h\in\dom A_{[B]}$.
Finally, we obtain from \eqref{h} that
\begin{equation}\label{res}
 (A_{[B]}-\lambda_+)h=(T-\lambda_+)h=(T-\lambda_+)(A_0-\lambda_+)^{-1}f=f,
\end{equation}
where again $\ran\gamma(\lambda_+)=\ker(T-\lambda_+)$ was used.
Hence $\ran (A_{[B]}-\lambda_+)=\cH$ holds.

\medskip
\noindent
\textit{Step 4.}
It follows from the symmetry of $A_{[B]}$ shown in Step~1 and the range condition
in Step~3 that the operator $A_{[B]}$ is self-adjoint in $\cH$.
The resolvent formula follows for $\lambda=\lambda_\pm$ immediately from the
identities~\eqref{h} and~\eqref{res} in Step~3.
Assume now that $\lambda\in\rho(A_{[B]})\cap\rho(A_0)$ is arbitrary.
We claim that the operator $I-BM(\lambda)$ is injective.
Indeed, if $\varphi\in\ker(I-BM(\lambda))$ then $\varphi\in\dom M(\lambda)=\ran\Gamma_0$
and hence $f\defeq\gamma(\lambda)\varphi\in\ker(T-\lambda)$, so that
$\Gamma_0 f=\varphi$. From
\begin{equation*}
  B\Gamma_1 f = BM(\lambda)\Gamma_0 f
  = B M(\lambda)\varphi = \varphi = \Gamma_0 f
\end{equation*}
we conclude that $f\in\dom A_{[B]}$ and hence $f\in\ker(A_{[B]}-\lambda)$.
Since $\lambda\in\rho(A_{[B]})$, we obtain $f=0$ and
$\varphi=\Gamma_0 f=0$.  Thus $I-BM(\lambda)$ is injective.

Next we show the inclusion
\begin{equation}\label{inclusion2}
  \ran\bigl(B\gamma(\overline\lambda)^*\bigr) \subset \ran\bigl(I-BM(\lambda)\bigr).
\end{equation}
To this end, let $\psi\in\ran(B\gamma(\ov\lambda)^*)$.  Then there exists an $f\in\cH$
such that $\psi=B\gamma(\ov\lambda)^* f$.  Set
\begin{align*}
  g \defeq& (A_{[B]}-\lambda)^{-1}f-(A_0-\lambda)^{-1}f \;\in\ker (T-\lambda), \\
  k \defeq& (A_{[B]}-\lambda)^{-1}f \;\in\dom A_{[B]}.
\end{align*}
From
\begin{align*}
  \Gamma_0 g &= \Gamma_0 k,\\
  \Gamma_1 g &= \Gamma_1 k-\Gamma_1(A_0-\lambda)^{-1}f
  = \Gamma_1 k-\gamma(\overline\lambda)^*f
\end{align*}
we conclude that
\begin{equation*}
  \bigl(I-BM(\lambda)\bigr)\Gamma_0 k
  = \Gamma_0 k- BM(\lambda)\Gamma_0 g
  = B \Gamma_1 k - B \Gamma_1 g
  = B\gamma(\overline\lambda)^*f = \psi.
\end{equation*}
This shows the inclusion in \eqref{inclusion2}. Now it follows in exactly
the same way as in Step~3 that for $\lambda\in\rho(A_{[B]})\cap\rho(A_0)$
the resolvent $(A_{[B]}-\lambda)^{-1}$ is given by the right-hand side
of \eqref{resformula}.

\medskip
\noindent
\textit{Step 5.}
Finally, assume that the second condition in (v) is satisfied,
i.e.\ that $A_1$ is self-adjoint.
Then $\ran M(\lambda_\pm)=\ran\Gamma_1$ by \cite[Proposition~2.6\,(iii)]{BL07}.
Hence, if $g\in\ran\Gamma_1$ then $g\in\dom B$ by (iv)
and $g \in \ran M (\lambda_\pm) \subset \ran (\ov{M(\lambda_\pm)})$.
Now (iii) implies that $Bg\in\ran\Gamma_0$.
This shows that the first condition in (v) is satisfied, and we can apply
Steps~1--4 of the proof.
\end{proof}

For the case when the spectrum of the self-adjoint operator $A_0$ does
not cover the whole real line a useful variant of Theorem~\ref{mainthm}
is formulated below.  Its proof is almost the same as the proof of Theorem~\ref{mainthm};
here the range condition in Step~3 of the proof needs only to be verified
for some real point in $\rho(A_0)$, which then automatically belongs to $\rho(A_{[B]})$.

\begin{theorem}\label{mainthm2}
Let $\{\cG,\Gamma_0,\Gamma_1\}$ be a quasi boundary triple for $T\subset S^*$
with corresponding $\gamma$-field $\gamma$ and Weyl function $M$.
Let $B$ be a linear operator in $\cG$ and assume that there
exists a $\lambda_0\in\rho(A_0)\cap\dR$ such that the following conditions are satisfied:
\begin{myenum}
 \item $B$ is symmetric;
 \item $1\in\rho(B\overline{M(\lambda_0)})$;
 \item $B\bigl(\ran\ov{M(\lambda_0)}\cap\dom B\bigr)\subset\ran\Gamma_0$;
 \item $\ran\Gamma_1\subset\dom B$;
 \item $B(\ran\Gamma_1)\subset\ran\Gamma_0$ \;\; or \;\; $\lambda_0\in\rho(A_1)$.
\end{myenum}
Then the operator
\begin{equation}\label{ab2}
  A_{[B]}f=Tf, \qquad
  \dom A_{[B]}=\bigl\{f\in\dom T:\Gamma_0 f= B\Gamma_1 f\bigr\},
\end{equation}
is a self-adjoint extension of $S$ such that $\lambda_0\in\rho(A_{[B]})$, and
\begin{equation}\label{resformula2}
  (A_{[B]}-\lambda)^{-1}=(A_0-\lambda)^{-1}
  + \gamma(\lambda)\bigl(I-BM(\lambda)\bigr)^{-1}B\gamma(\overline\lambda)^*
\end{equation}
holds for all $\lambda\in\rho(A_{[B]})\cap\rho(A_0)$.
\end{theorem}

For completeness the corresponding version of Corollary~\ref{maincor} is also stated.

\begin{corollary}\label{maincor2}
Let $\{\cG,\Gamma_0,\Gamma_1\}$ be a quasi boundary triple for $T\subset S^*$
with corresponding $\gamma$-field $\gamma$ and Weyl function $M$.
Let $B$ be a bounded self-adjoint operator in $\cG$ and assume that there
exists a $\lambda_0\in\rho(A_0)\cap\dR$ such that the following conditions are satisfied:
\begin{myenum}
 \item
 $1\in\rho(B\overline{M(\lambda_0)})$;
 \item
 $B(\ran\ov{M(\lambda_0)})\subset\ran\Gamma_0$;
 \item
 $B(\ran\Gamma_1)\subset\ran\Gamma_0$ \;\; or \;\; $\lambda_0\in\rho(A_1)$.
\end{myenum}
Then the operator $A_{[B]}$ in \eqref{ab2} is a self-adjoint extension of $S$
such that $\lambda_0\in\rho(A_{[B]})$,
and the resolvent formula \eqref{resformula2} holds for
all $\lambda\in\rho(A_{[B]})\cap\rho(A_0)$.
\end{corollary}

In the following theorem we consider the situation that the values of the
Weyl function are bounded operators which
tend to zero as $\lambda \to -\infty$.
In order to formulate this theorem, let us introduce the following notation.
For a self-adjoint operator $B$ with spectral measure $E_B(\cdot)$
we define its positive and negative parts by
\begin{align}\label{eq:negPos}
  B_\pm \defeq \pm\int_{\dR_\pm} \lambda\:\drm E_B (\lambda),
\end{align}
respectively, so that $B_\pm \geq 0$ and $B = B_+ - B_-$.
If $A_0$ is bounded from below, the assumption $M(\lambda)\to0$ as $\lambda\to-\infty$
implies that $(M (\lambda) \phi, \phi) \to 0$ as $\lambda \to - \infty$ for each $\phi \in \ran \Gamma_0$. Recall from Lemma~\ref{le:Mnonneg} that  this implies non-negativity of $M(\lambda)$ for $\lambda<\min\sigma(A_0)$; in particular, under these conditions $\ov{M(\lambda)}^{1/2}$ is well defined
for such $\lambda$.

\begin{theorem}\label{thm:lowerBound}
Let $\{\cG,\Gamma_0,\Gamma_1\}$ be a quasi boundary triple for $T\subset S^*$
with corresponding $\gamma$-field $\gamma$ and Weyl function $M$.
Assume that $A_0$ is bounded from below
and that $M(\lambda)$ is bounded for one (and hence for all) $\lambda\in\rho(A_0)$.
Let $B$ be a self-adjoint operator in $\cG$ which is bounded from above
and assume that the following conditions are satisfied:
\begin{myenum}
\item
  $\bigl\|\ov{M(\lambda)}\bigr\|\to0$ \; as \; $\lambda\to-\infty$;
\item
  $\ran\ov{M(\lambda)}^{1/2}\subset\dom B$ \;\; for all\; $\lambda<\min\sigma(A_0)$;
\item
  $B\bigl(\ran\ov{M(\lambda)}\bigr)\subset\ran\Gamma_0$ \;\; for all\; $\lambda<\min\sigma(A_0)$;
\item
  $\ran\Gamma_1\subset\dom B$;
\item
  $B(\ran\Gamma_1)\subset\ran\Gamma_0$ \quad or \quad $A_1$ is self-adjoint and bounded from below.
\end{myenum}
Then the operator
\begin{equation*}
  A_{[B]}f=Tf, \qquad
  \dom A_{[B]}=\bigl\{f\in\dom T:\Gamma_0 f= B\Gamma_1 f\bigr\},
\end{equation*}
is a self-adjoint extension of $S$, which is bounded from below,
and the resolvent formula \eqref{resformula2} holds for
all $\lambda\in\rho(A_{[B]})\cap\rho(A_0)$.

Moreover, the following statements are true.

\begin{myenuma}
\item
If $B\le0$, then $\min\sigma(A_{[B]})\ge\min\sigma(A_0)$.
\item
If there exist $\alpha\in(0,1]$, $\mu\le\min\sigma(A_0)$ and $C>0$ such that
\begin{equation}\label{decay}
  \bigl\|\overline{M(\lambda)}\bigr\| \le \frac{C}{\bigl(\mu-\lambda \bigr)^\alpha},
  \qquad \text{for}\;\;\lambda < \mu,
\end{equation}
then the lower bound of $A_{[B]}$ satisfies
\begin{equation}\label{Bbound}
  \min\sigma(A_{[B]}) \ge \mu - (C\|B_+\|)^{1/\alpha}.
\end{equation}
\end{myenuma}
\end{theorem}

\begin{proof}
Assumption (i) and the boundedness of $B_+$ imply that there exists
a $\mu_0<\min\sigma(A_0)$ such that
\begin{equation}\label{BplMle1}
  \bigl\|B_+\ov{M(\lambda)}\bigr\| < 1
\end{equation}
for all $\lambda\le\mu_0$.
In the following let $\lambda \leq \mu_0$.  Since
\[
  \sigma\bigl(\ov{M(\lambda)}^{1/2}B_+\ov{M(\lambda)}^{1/2}\bigr)\setminus\{0\}
  = \sigma\bigl(B_+\overline{M(\lambda)}\bigr)\setminus\{0\},
\]
relation \eqref{BplMle1} yields that
\begin{equation}\label{MBplMdelta}
  \sigma\bigl(\ov{M(\lambda)}^{1/2}B_+\ov{M(\lambda)}^{1/2}\bigr)
  \subset [-\beta,\beta]
\end{equation}
for some $\beta\in(0,1)$.
It follows from assumption (ii), the relation $\dom B_-=\dom B$ and the closed graph theorem
that
\[
 B_-\ov{M(\lambda)}^{1/2}
\]
is a bounded everywhere defined operator. Hence
\[
  \ov{M(\lambda)}^{1/2}B_-\ov{M(\lambda)}^{1/2}
\]
is a bounded, non-negative operator.
This, together with \eqref{MBplMdelta}, shows that
\[
  \sigma\bigl(\ov{M(\lambda)}^{1/2}B\ov{M(\lambda)}^{1/2}\bigr)
  \subset (-\infty,\beta];
\]
see, e.g.\ \cite[Lemma~3 in \S 9.4]{BS87}.
In particular, we have
\begin{equation}\label{oneinresolv}
  1 \in \rho\bigl(\ov{M(\lambda)}^{1/2}B\ov{M(\lambda)}^{1/2}\bigr).
\end{equation}
Since by the closed graph theorem the operator  $B\ov{M(\lambda)}^{1/2}$ is bounded,
it follows from
\[
  \sigma\bigl(\ov{M(\lambda)}^{1/2}B\ov{M(\lambda)}^{1/2}\bigr)\setminus\{0\}
  = \sigma\bigl(B\overline{M(\lambda)}\bigr)\setminus\{0\},
\]
and \eqref{oneinresolv}
that $1\in\rho(B\ov{M(\lambda)})$, i.e.\ condition (ii) in Theorem~\ref{mainthm2}
is satisfied for $\lambda_0 = \lambda$.  Moreover, conditions (ii)--(v) of the current theorem
imply conditions (iii)--(v) of Theorem~\ref{mainthm2}.
The latter theorem yields that $A_{[B]}$ is a self-adjoint extension of $S$
and that $\lambda\in\rho(A_{[B]})$ for all $\lambda\le\mu_0$, which shows that
$A_{[B]}$ is bounded from below.

(a)
Assume that $B\le0$ and let $\lambda<\min\sigma(A_0)$.  Then
\[
  \ov{M(\lambda)}^{1/2}B\ov{M(\lambda)}^{1/2} \le 0
\]
and hence \eqref{oneinresolv} is satisfied.  Therefore $1\in\rho(B\ov{M(\lambda)})$
and, as above, one concludes that $\lambda\in\rho(A_{[B]})$.
Hence $\min\sigma(A_{[B]})\ge\min\sigma(A_0)$.

(b)
Now assume that \eqref{decay} is satisfied and let
\[
  \lambda < \mu-\bigl(C\|B_+\|\bigr)^{1/\alpha}.
\]
Then
\[
  \bigl\|B_+\ov{M(\lambda)}\bigr\| \le \bigl\|B_+\bigr\|\,\bigl\|\ov{M(\lambda)}\bigr\|
  < \frac{(\mu-\lambda)^\alpha}{C}\cdot\frac{C}{(\mu-\lambda)^\alpha} = 1,
\]
i.e.\ \eqref{BplMle1} is satisfied.  The first part of the proof shows
that $\lambda\in\rho(A_{[B]})$, which proves \eqref{Bbound}.
\end{proof}

Theorem~\ref{thm:lowerBound} implies the
following asymptotic estimates on the lower bound of the extensions $A_{[\omega B]}$
of $S$ with $\omega \in \dR_+$ as $\omega\to+\infty$ and $\omega\to 0+$.

\begin{corollary}\label{cor:Asymptotics}
Let $B$ be a self-adjoint operator in $\cG$ which is bounded from above and let the assumptions {\rm(i)--(v)} and \eqref{decay} from Theorem~\ref{thm:lowerBound}
be satisfied.  Then the operators
\begin{equation*}
  A_{[\omega B]}f = Tf, \qquad
  \dom A_{[\omega B]} = \bigl\{f\in\dom T:\Gamma_0 f = \omega B \Gamma_1 f\bigr\}, \qquad
  \omega \geq 0,
\end{equation*}
are self-adjoint extensions of $S$, which are bounded from below.
Define the function $F(\omega) \defeq \min\sigma(A_0) - \min\sigma(A_{[\omega B]})$
and let $F^+$ be its positive part.
Then the following asymptotic estimates are satisfied:
\begin{myenuma}
\item $\bigl|\min\sigma(A_{[\omega B]})\bigr| = \cO(|\omega|^{1/\alpha})$
as $\omega \to + \infty$;
\item if the bound \eqref{decay} holds for $\mu = \min\sigma(A_0)$, then
$F^+(\omega) = \cO(|\omega|^{1/\alpha})$ as  $\omega \to 0 +$.
\end{myenuma}
\end{corollary}

\begin{proof}
The asymptotic estimate in (a) follows directly from \eqref{Bbound} and from
the fact that $\min\sigma(A_{[\omega B]})\le \min\sigma(A_{\rm F}) <\infty$,
where $A_{\rm F}$ is the Friedrichs extension of $S$.
The asymptotic estimate in (b) is again a straightforward consequence
of~\eqref{Bbound}.
\end{proof}

% *******************************************************************
\section{Elliptic differential operators on domains with \\
non-compact boundaries}\label{sec:PDE}
% *******************************************************************

\noindent
In this section we apply the abstract results from Section~\ref{sec2} to
second-order elliptic differential operators on unbounded domains
with non-compact boundaries. We refer the reader to \cite{BL07,BL12}
and \cite{BLL13IEOT,BM14} for related results for bounded and exterior domains, respectively.
Here we shall rely on classical results on the $H^2$-regularity of the
corresponding Dirichlet and Neumann realizations, and make use of
a set of assumptions that can be found in a more general context in \cite{B65}.

Let $\Omega$ be a domain in $\RR^n$ which is \emph{uniformly regular} in the sense
of~\cite[page~72]{F67}; see also \cite{B65,B61}.
This means that $\partial\Omega$ is $C^\infty$-smooth and that
there exists a covering of $\Omega$ by open sets $\Omega_j$, $j\in\dN$, and $n_0\in\dN$
such that at most $n_0$ of the $\Omega_j$ have a non-empty intersection,
and
a family of $C^\infty$-homeomorphisms
\[
  \varphi_j:\Omega_j\cap\Omega\rightarrow \cB_1\cap\{x_n>0\}, \qquad
  \text{where}\quad \cB_r=\{x\in\RR^n:\Vert x\Vert < r\},
\]
such that $\varphi_j:\Omega_j\cap\partial\Omega\rightarrow \cB_1\cap\{x_n=0\}$,
the derivatives of $\varphi_j$, $j\in\dN$, and their inverses are uniformly bounded, and
$\bigcup_j\varphi^{-1}_j(\cB_{1/2})$ covers a uniform neighbourhood of $\partial\Omega$.
Note that these assumptions are automatically fulfilled, e.g.\
for domains with compact $C^\infty$-smooth boundaries or for compact,
smooth perturbations of half-spaces.
Let
\begin{equation*}
  \cL  = - \sum_{j, k = 1}^n \frac{\partial}{\partial x_j} a_{jk}
  \frac{\partial}{\partial x_k} + a
\end{equation*}
be a differential expression on $\Omega$ with bounded
coefficients $a_{j k}\in C^\infty(\overline\Omega)$
satisfying $a_{jk} (x) = \overline{a_{k j} (x)}$ for all $x \in \overline{\Omega}$,
and having bounded, uniformly continuous derivatives on $\overline\Omega$;
cf.\ \cite[(S1)--(S5) in Chapter~4]{B65}.
Moreover, it is assumed that $a\in L^\infty(\Omega)$ is real-valued
and that $\cL$ is uniformly elliptic, i.e.\ there exists an $E > 0$ such that
\begin{equation}\label{ellipticity}
  \sum_{j,k=1}^n a_{jk}(x)\xi_j\xi_k \ge E \sum_{k=1}^n\xi_k^2, \qquad
  \xi=(\xi_1,\dots\xi_n)^\top \in \dR^n,\;\; x \in \overline \Omega.
\end{equation}

In the following we denote by $H^s (\Omega)$ and $H^s (\partial \Omega)$ the
Sobolev spaces of order $s \geq 0$ on $\Omega$ and
$\partial \Omega$, respectively.
For $f \in C_0^\infty(\ov\Omega)\defeq\{g|_{\ov\Omega}:g\in C_0^\infty(\RR^n)\}$ let
\begin{align*}
  \frac{\partial f}{\partial \nu_\cL} \Big|_{\partial \Omega}
  \defeq \sum_{j,k = 1}^n a_{jk} \nu_j \frac{\partial f}{\partial x_k}\Big|_{\partial\Omega}
\end{align*}
be the co-normal derivative of $f$ at $\partial \Omega$ with respect to~$\cL$,
where $\nu = (\nu_1, \dots, \nu_n)^\top$ is the unit normal vector field
at $\partial \Omega$ pointing outwards.
Then Green's identity
\begin{align}\label{eq:Green}
  (\cL f, g) - (f, \cL g) = \left( f |_{\partial \Omega}, \frac{\partial g}{\partial \nu_\cL} \Big|_{\partial \Omega} \right)
  - \left(\frac{\partial f}{\partial \nu_\cL} \Big|_{\partial \Omega}, g |_{\partial \Omega} \right)
\end{align}
holds for all $f, g \in C_0^\infty (\overline \Omega)$
(see, e.g.\ \cite[Theorem~4.4]{F67}); here the inner products
on the left-hand side are in $L^2(\Omega)$ and the inner
products on the right-hand side are in $L^2 (\partial \Omega)$.
Recall that the mapping
\begin{align*}
 f \mapsto \left\{f |_{\partial \Omega}, \frac{\partial f}{\partial \nu_\cL} \Big|_{\partial \Omega} \right\},
 \qquad f\in C_0^\infty (\overline{\Omega}),
\end{align*}
extends by continuity to a bounded, surjective map from $H^2 (\Omega)$
to $H^{3/2} (\partial \Omega) \times H^{1/2} (\partial \Omega)$,
and that Green's identity~\eqref{eq:Green} extends to all $f, g \in H^2 (\Omega)$;
see, e.g.\ \cite[Theorem~3.9]{F67}. For the extended trace and normal derivative
we write again $f |_{\partial \Omega}$ and
$\frac{\partial f}{\partial \nu_\cL} |_{\partial \Omega}$, respectively.

In order to construct a quasi boundary triple, consider the operators $S$ and $T$
in $L^2(\Omega)$ given by
\begin{align}\label{eq:Selliptic}
  S f = \cL f, \qquad
  \dom S = \left\{ f \in H^2 (\Omega) : f |_{\partial \Omega}
  = \frac{\partial f}{\partial \nu_\cL} \Big|_{\partial \Omega} = 0 \right\},
\end{align}
and
\begin{align}\label{eq:Telliptic}
  T f = \cL f, \qquad \dom T = H^2 (\Omega).
\end{align}
The proof of the following proposition is similar to the proof
of~\cite[Proposition~4.6]{BL07} and is omitted.
We only mention that the self-adjointness of $A_{\rm N}$ in (ii)
follows from \cite[Theorem~7.1\,(a) and Theorem~7.2]{B65}.

\begin{proposition}\label{prop:QBTelliptic}
The operator $S$ in~\eqref{eq:Selliptic} is closed, symmetric and densely defined
with $\overline{T} = S^*$, and the triple
$\{L^2 (\partial \Omega), \Gamma_0, \Gamma_1\}$ with
\[
  \Gamma_0 f = \frac{\partial f}{\partial\nu_\cL}\Big|_{\partial\Omega}, \qquad
  \Gamma_1 f = f|_{\partial\Omega}, \qquad f \in \dom T,
\]
is a quasi boundary triple for $T\subset S^*$ with the following properties:
\begin{myenum}
  \item
    $\ran(\Gamma_0,\Gamma_1)^\top = H^{1/2}(\partial\Omega) \times H^{3/2} (\partial\Omega)$;
  \item
    $A_0 = T \upharpoonright \ker \Gamma_0$ coincides with the
    self-adjoint \emph{Neumann operator}
    \begin{align*}
      A_{\rm N} f = \cL f, \qquad \dom A_{\rm N} = \left\{ f \in H^2 (\Omega) :
      \frac{\partial f}{\partial \nu_\cL} \Big|_{\partial \Omega} = 0 \right\},
     \end{align*}
    and $A_1 = T \upharpoonright \ker \Gamma_1$ coincides with the
    self-adjoint \emph{Dirichlet operator}
    \begin{align*}
      A_{\rm D} f = \cL f, \qquad \dom A_{\rm D} = \bigl\{ f \in H^2 (\Omega) :
      f|_{\partial \Omega} = 0 \bigr\};
    \end{align*}
    in particular, $A_0$ and $A_1$ are bounded from below by $\essinf a$.
\end{myenum}
\end{proposition}

\medskip

\noindent
In the next proposition we prove a couple of properties of the Weyl function,
which turns out to be the Neumann-to-Dirichlet map.
These properties are needed in order to apply the results from Section~\ref{sec2}.
In particular, in (iv) we prove a decay estimate for the Weyl function,
whose proof is based on an argument due to S.~Agmon \cite[Section~2]{A62};
related methods were also used in the proof of \cite[Theorem~4.1]{AGW14}.
This estimate can also be shown using the calculus of parameter-dependent
pseudo-differential operators provided in, e.g.\ \cite[Chapter~2]{G96}.

\begin{proposition}\label{prop:Weyl_elliptic}
Let $\{L^2(\partial\Omega),\Gamma_0,\Gamma_1\}$
be the quasi boundary triple from Proposition~\ref{prop:QBTelliptic}
and let $M$ be the corresponding Weyl function.
\begin{myenum}
  \item
    The function~$M$ is given by the \emph{Neumann-to-Dirichlet map}, i.e.\
    it satisfies
    \begin{align*}
      M(\lambda)\frac{\partial f}{\partial \nu_\cL} \Big|_{\partial \Omega}
      = f |_{\partial \Omega}, \qquad
      f \in \ker(T-\lambda),\;\; \lambda \in \rho(A_{\rm N}).
    \end{align*}
    Moreover,
    \begin{equation}\label{eq:spaces}
      \dom M (\lambda) = H^{1/2} (\partial \Omega)\quad\text{and}\quad
      \ran M (\lambda) \subset H^{3/2} (\partial \Omega)
    \end{equation}
    for each $\lambda \in \rho (A_{\rm N})$. If $\lambda < \min \sigma (A_{\rm N})$
    then $\ran M(\lambda) = H^{3/2} (\partial \Omega)$ and $\ker M(\lambda)=\{0\}$.
  \item For all $\lambda \in \rho (A_{\rm N})$ the operator $M(\lambda)$ is bounded
    and non-closed in $L^2 (\partial \Omega)$, and its closure
    satisfies $\ran \overline{M (\lambda)} \subset H^1(\partial\Omega)$.
  \item
    For $\lambda<\min\sigma(A_{\rm N})$ the operator $M(\lambda)$ is non-negative
    and satisfies
    \begin{equation}\label{eq:MRanSquareroot}
      \ran \overline{M (\lambda)}^{1/2} = H^{1/2}(\partial\Omega).
    \end{equation}
  \item
    For each $\mu<\min\sigma(A_{\rm N})$ there exists a
    constant $C=C(\cL,\Omega,\mu)$ such that
    \begin{equation}\label{estimate}
       \big\| \overline{M (\lambda)} \big\| \leq \frac{C}{( \mu - \lambda)^{1/2}}\,, \qquad
       \lambda < \mu.
    \end{equation}
\end{myenum}
\end{proposition}

\begin{proof}
(i) The representation of the Weyl function and the assertion~\eqref{eq:spaces}
follow directly from the definition of the boundary maps and the Weyl function.
Since $A_{\rm D}$ is the Friedrichs extension of $S$,
each $\lambda < \min \sigma (A_{\rm N})$ belongs to $\rho (A_{\rm D})$.
Thus the Dirichlet-to-Neumann map
$f |_{\partial \Omega} \mapsto \frac{\partial f}{\partial \nu_\cL} |_{\partial \Omega}$,
$f \in \ker (T - \lambda)$, is well defined with domain equal to $H^{3/2} (\partial \Omega)$.
Since its inverse is given by~$M (\lambda)$,
we have $\ran M (\lambda) = H^{3/2} (\partial \Omega)$ and $\ker M(\lambda)=\{0\}$.

(ii) These properties can be shown in the same way as in \cite[Lemma~4.4]{BLL13IEOT}.

(iii)
Let $\fra$ be the quadratic form corresponding to the Neumann operator, i.e.\
\begin{equation}\label{defa}
  \begin{aligned}
    \fra[f] &\defeq \fra_1[f]+(af,f) \\[1ex]
    &\defeq \int_\Omega \biggl(\sum_{j,k = 1}^n a_{jk}
    \frac{\partial f}{\partial x_k} \cdot \frac{\partial\ov f}{\partial x_j}
    \biggr)\drm x + (af,f),
    \qquad f\in\dom\fra \defeq H^1(\Omega),
  \end{aligned}
\end{equation}
where $(\cdot, \cdot)$ is the inner product in~$L^2 (\Omega)$.
Let $\lambda<\min\sigma(A_{\rm N})$, $\varphi\in H^{1/2}(\partial\Omega)$
and let $f \in \ker(T-\lambda)$ such that
$\varphi = \frac{\partial f}{\partial\nu_\cL}\big|_{\partial\Omega}$,
i.e.\ $f=\gamma(\lambda)\varphi$.
Then Green's first identity yields
\begin{align}\label{eq:MEllipticNonnegative}
  \bigl(M(\lambda)\phi,\phi\bigr)
  &= \biggl(f|_{\partial\Omega},
  \frac{\partial f}{\partial\nu_\cL}\Big|_{\partial\Omega}\biggr)
  = - (f,\cL f) + \fra[f] \notag \\
  & \ge \bigl(-\lambda + \min \sigma(A_{\rm N})\bigr)\|f\|^2_{L^2(\Omega)}
  \ge 0,
\end{align}
which shows that $M(\lambda)$ is a non-negative operator.

Next we show \eqref{eq:MRanSquareroot}.
Let $\lambda < \min \sigma(A_{\rm N})$ and consider the quadratic
form in $L^2(\partial\Omega)$ defined by
\begin{align*}
  \frt_\lambda[\phi] \defeq \bigl(M(\lambda)^{-1}\phi,\phi\bigr), \qquad
  \dom\frt_\lambda = H^{3/2}(\partial\Omega),
\end{align*}
which is well defined by item (i).
The form $\frt_\lambda$ is densely defined, symmetric
and non-negative by \eqref{eq:MEllipticNonnegative}.
There exist constants $\wt c_1, \wt c_2>0$ such that
\begin{equation}\label{eq:equivNorm}
  \wt c_1 \|f\|_{H^1(\Omega)}^2 \le \fra [f] - \lambda\|f\|_{L^2(\Omega)}^2
  \le \wt c_2 \|f\|_{H^1(\Omega)}^2, \qquad f \in H^1(\Omega);
\end{equation}
to see the first inequality, set $a_0\defeq\essinf a$, $\sigma_0\defeq\min\sigma(A_{\rm N})$
and let $\eps>0$.
Then (where we write $\|\cdot\|$ for $L^2$-norms and use $E$ from \eqref{ellipticity})
\begin{align*}
  \fra[f]-\lambda\|f\|^2
  &= \eps\fra_1[f] + \eps(af,f) + (1-\eps)\fra[f] - \lambda\|f\|^2 \\[1ex]
  &\ge \eps E\|\nabla f\|^2 + \eps a_0\|f\|^2 + (1-\eps)\sigma_0\|f\|^2 - \lambda\|f\|^2 \\[1ex]
  &\ge \min\bigl\{\eps E,\eps a_0+(1-\eps)\sigma_0-\lambda\bigr\}\|f\|_{H^1(\Omega)}^2.
\end{align*}
If $\eps$ is small enough, then the last minimum is a positive number.
Hence the first inequality in \eqref{eq:equivNorm} is shown.
The second inequality follows easily from the boundedness of the coefficients.

For each $\varphi\in H^{3/2}(\partial\Omega)$ there exists $f\in\ker(T-\lambda)$
such that $f|_{\partial\Omega} = \phi$.  Similarly to \eqref{eq:MEllipticNonnegative}
one obtains
\begin{align*}
  \frt_\lambda[\phi] &= \bigl(M(\lambda)^{-1}\varphi,\varphi\bigr)
  = \biggl(\frac{\partial f}{\partial\nu_\cL}\Big|_{\partial\Omega},
  f|_{\partial\Omega}\biggr)
  \\[1ex]
  &= -(\cL f,f)+\fra[f] = \fra[f] - \lambda\|f\|_{L^2(\Omega)}^2, \quad \varphi \in H^{3/2} (\partial \Omega).
\end{align*}
Together with \eqref{eq:equivNorm} this yields
\[
  \wt c_1\|f\|_{H^1(\Omega)}^2 \le \frt_\lambda[\varphi]
  \le \wt c_2\|f\|_{H^1(\Omega)}^2
\]
for all $\varphi \in H^{3/2} (\partial \Omega)$ and corresponding $f \in \ker (T - \lambda)$
with $f |_{\partial \Omega} = \varphi$. Since the trace map provides an isomorphism
from $\{g \in H^1(\Omega): \cL g = \lambda g\}$
onto $H^{1/2}(\partial \Omega)$, it follows that there exist $c,C>0$ such that
\begin{align*}
  c \|\phi\|_{H^{1/2}(\partial\Omega)}^2
  \le \frt_\lambda [\phi]
  \le C \|\phi\|_{H^{1/2}(\partial\Omega)}^2, \qquad
  \phi \in H^{3/2}(\partial\Omega).
\end{align*}
Hence the domain of the closure of $\frt_\lambda$
equals $H^{1/2} (\partial \Omega)$.
From this we obtain \eqref{eq:MRanSquareroot} because $\overline{M(\lambda)}^{\,-1}$
is the self-adjoint operator that corresponds to the closure of $\frt_\lambda$.

(iv)
Let $\dS^1=\dR/(2\pi\dZ)$ be the one-dimensional torus and consider the
product $\Omega \times \dS^1$. On this manifold we consider
the elliptic differential expression
\begin{equation*}
\label{cLS}
  \cL_{\dS} = \cL - \frac{\partial ^2}{\partial t^2},
\end{equation*}
where $t$ denotes the variable in $\dS^1$ and $\cL$ acts with respect to the
variable $x \in \Omega$. The manifold $\partial \Omega \times \dS^1$ is the
boundary of $\Omega \times \dS^1$. For $s \geq 0$ let $H^s(\Omega \times \dS^1)$
and $H^s(\partial \Omega \times \dS^1)$ be the Sobolev spaces on $\Omega \times \dS^1$
and $\partial \Omega \times \dS^1$, respectively.  On $H^2 (\Omega \times \dS^1)$
we can define traces $f|_{\partial \Omega \times \dS^1}$
and normal derivatives
$\frac{\partial f}{\partial \nu_{\cL_\dS}} |_{\partial \Omega \times \dS^1}$.
Note that, for functions of the form $f(x,t)=g(x)h(t)$ with $g\in H^2 (\Omega)$
and $h\in C^\infty(\dS^1)$, we have
\begin{equation}\label{trprod}
  f |_{\partial \Omega \times \dS^1} = h\cdot g |_{\partial \Omega} \qquad \text{and} \qquad
  \frac{\partial f}{\partial \nu_{\cL_\dS}} \big|_{\partial \Omega \times \dS^1}
  = h \cdot \frac{\partial g}{\partial \nu_\cL} \big|_{\partial \Omega}.
\end{equation}
Next let us introduce the operator
\begin{align*}
  T_\dS f &= \cL_\dS f, \qquad \dom T_{\dS} = H^2 (\Omega \times \dS^1),
\end{align*}
in the space $L^2(\Omega \times \dS^1)$ and the
triple $\{L^2(\partial \Omega \times \dS^1), \Gamma_0^\dS,\Gamma_1^\dS\}$ where
\[
  \Gamma_0^\dS f = \frac{\partial f}{\partial \nu_{\cL_\dS}} \big|_{\partial \Omega \times \dS^1}
  \quad \text{and} \quad
  \Gamma_1^\dS f = f |_{\partial \Omega \times \dS^1}.
\]
Similarly to Proposition~\ref{prop:QBTelliptic} one verifies
that $\{L^2(\partial \Omega \times \dS^1), \Gamma_0^\dS, \Gamma_1^\dS\}$ is a
quasi boundary triple for $T_\dS \subset S_\dS^*$,
where $S_{\dS} = T_\dS \upharpoonright (\ker \Gamma_0^\dS \cap \ker \Gamma_1^\dS)$
and $A_{\rm N}^\dS \defeq T_{\dS} \upharpoonright \ker \Gamma_0^\dS$.
It follows from \eqref{trprod} that for $f(x,t)=g(x)h(t)$ with $g\in H^2 (\Omega)$
and $h\in C^\infty(\dS^1)$ we have
\begin{equation*}
  \Gamma_j^\dS f = h\, \Gamma_j g, \qquad j=0,1.
\end{equation*}
By the trace theorem we have
\[
  \ran \Gamma_0^\dS = H^{1/2}(\partial \Omega \times \dS^1) \quad \text{and} \quad
  \ran \Gamma_1^\dS = H^{3/2} (\partial \Omega \times\dS^1),
\]
and, as for Proposition~\ref{prop:QBTelliptic}, one shows that the values of the
Weyl function $M_\dS$ corresponding to the
quasi boundary triple $\{L^2(\partial \Omega \times \dS^1), \Gamma_0^\dS, \Gamma_1^\dS\}$
are bounded (non-closed) operators in $L^2(\partial\Omega\times\dS^1)$
with $\ran \overline{M_\dS (\lambda)} \subset H^1 (\partial \Omega \times \dS^1)$;
hence $\overline{M_\dS (\lambda)}$ can be regarded as a bounded operator
from $L^2(\partial \Omega \times \dS^1)$ to $H^1(\partial \Omega \times \dS^1)$
for $\lambda\in\rho(A_{\rm N}^\dS)$ by the closed graph theorem.

It is not difficult to see that the operator
$A_{\rm N}^\dS$ is bounded from below with
$\min \sigma (A_{\rm N}^\dS) = \min \sigma (A_{\rm N})$. Let $\mu < \min \sigma (A_{\rm N})$.
In the following we consider the case
\begin{equation}\label{case1}
  \lambda<\mu-1<\mu <\lambda_0 < \min \sigma(A_{\rm N}),
\end{equation}
where $\lambda_0$ is some fixed constant, and we set
\begin{equation}\label{m}
  m \defeq \sup_{\lambda\leq\mu-1}\frac{\sqrt{\mu-\lambda}}{\sqrt{\lambda_0-\lambda}-1} <\infty.
\end{equation}
By the above considerations there exists $C' > 0$, depending
on~$\lambda_0$, $\Omega$ and $\cL$, with
\begin{align}\label{eq:MSbdd}
  \|M_\dS(\lambda_0)\psi\|_{H^1(\partial\Omega\times\dS^1)}
  \le C'\|\psi\|_{L^2(\partial\Omega\times\dS^1)}, \qquad
  \psi \in H^{1/2}(\partial\Omega\times\dS^1).
\end{align}
For an arbitrary $\phi \in H^{1/2} (\partial \Omega)$ and $k \in \dN$ let us define
\begin{align*}
  f (x,t) \defeq e^{ikt}\bigl(\gamma(\lambda_0-k^2)\phi\bigr)(x), \qquad
  x \in \Omega,\; t \in \dS^1.
\end{align*}
Then $f \in \dom T_\dS$ and
\begin{align*}
  \cL_\dS f = (\lambda_0-k^2)f + k^2 f = \lambda_0 f,
\end{align*}
that is, $f \in \ker (T_\dS - \lambda_0)$.  Moreover,
\begin{align*}
  (\Gamma_0^\dS f) (x, t) = e^{ikt}\bigl(\Gamma_0\gamma(\lambda_0-k^2)\phi\bigr)(x)
  = e^{ikt}\phi(x).
\end{align*}
Hence, setting $\psi_{k,\phi}(x,t) \defeq e^{ikt}\phi(x)$ for $t\in\dS^1$
and $x \in \partial \Omega$ we have
\begin{align*}
  \bigl(M_\dS(\lambda_0)\psi_{k,\phi}\bigr)(x,t)
  &= \bigl(M_\dS(\lambda_0)\Gamma_0^\dS f\bigr)(x,t)
  = (\Gamma_1^\dS f)(x,t) \\
  & = e^{ikt}\bigl(\Gamma_1\gamma(\lambda_0-k^2)\phi\bigr)(x)
  = e^{ikt}\bigl(M(\lambda_0-k^2)\phi\bigr)(x).
\end{align*}
It follows that
\begin{align*}
  \|M_\dS(\lambda_0)\psi_{k,\phi}\|_{H^1(\partial\Omega\times\dS^1)}^2
  &\ge \left\|\frac{\partial}{\partial t}e^{ikt}M(\lambda_0-k^2)
  \phi \right\|_{L^2(\partial\Omega\times\dS^1)}^2 \\
  &=  2\pi k^2 \|M(\lambda_0-k^2)\phi\|_{L^2(\partial\Omega)}^2.
\end{align*}
Combining this estimate with~\eqref{eq:MSbdd} we obtain
\begin{align*}
  \|M(\lambda_0-k^2)\phi\|_{L^2(\partial\Omega)}
  &\le \frac{1}{\sqrt{2\pi}\,k}
  \|M_\dS(\lambda_0)\psi_{k,\phi}\|_{H^1(\partial\Omega\times\dS^1)} \\
  &\le \frac{1}{\sqrt{2\pi}\,k}\,C'\|\psi_{k,\phi}\|_{L^2(\partial\Omega\times\dS^1)}
  = \frac{C'}{k}\|\phi\|_{L^2(\partial\Omega)}
\end{align*}
for all $\phi \in H^{1/2}(\partial\Omega)$ and all $k\in\dN$.
As $\lambda < \lambda_0-1$ by \eqref{case1}, there exists $k \in \NN$
such that $\lambda_0-(k+1)^2 \le \lambda < \lambda_0-k^2$.
Since $\lambda\mapsto(M(\lambda)\varphi,\varphi)$ is non-decreasing
on $(-\infty,\min\sigma(A_{\rm N}))$ for every $\varphi\in H^{1/2}(\partial\Omega)$
by Lemma~\ref{le:Mnonneg}\,(i)
and $M(\lambda)\ge0$ by item (ii), also $\lambda\mapsto\bigl\|\ov{M(\lambda)}\bigr\|$
is non-decreasing on the same interval.  Hence
\begin{align*}
  \big\|\overline{M(\lambda)}\big\| \le \big\|\overline{M(\lambda_0-k^2)}\big\|
  \le \frac{C'}{k} \le \frac{C'}{\sqrt{\lambda_0-\lambda}-1}
  \le \frac{m\,C'}{\sqrt{\mu-\lambda}\,}
\end{align*}
for all $\lambda < \mu-1$, where the constant $m$ in \eqref{m} was used in the
last estimate.  It remains to show the estimate in \eqref{estimate}
for $\lambda\in[\mu-1,\mu)$.  Here the monotonicity of $M$ yields
\begin{equation*}
  \big\|\overline{M(\lambda)}\big\| \le \big\|\overline{M(\mu)}\big\|
  \le \frac{1}{\sqrt{\mu-\lambda}}\big\|\overline{M(\mu)}\big\|, \qquad
  \lambda\in [\mu-1,\mu),
\end{equation*}
and hence we have shown \eqref{estimate} with
$C \defeq \max\bigl\{m\,C',\|\overline{M(\mu)}\|\bigr\}$.
\end{proof}

\begin{remark}\label{rem:Konstante}
A possible choice of the constant $C = C(\cL,\Omega,\mu)$ can be read off
from the proof of the preceding proposition, namely
\begin{align*}
  C = \max \left\{\sup_{\lambda \le \mu-1}
  \frac{\sqrt{\mu-\lambda}}{\sqrt{\lambda_0-\lambda}-1}\big\|
  \overline{M_\dS(\lambda_0)}\big\|,
  \big\|\overline{M(\mu)}\big\|\right\}
\end{align*}
with $\lambda_0 \in(\mu,\min\sigma(A_{\rm N}))$, where $M_\dS(\lambda_0)$
is the Neumann-to-Dirichlet map for the differential
expression $\cL - \frac{\partial^2}{\partial t^2} - \lambda_0$ on $\Omega \times \dS^1$,
considered as an operator from the space $L^2 (\partial \Omega \times \dS^1)$ to $H^1 (\partial \Omega \times \dS^1)$.
\end{remark}

\begin{remark}
In general the assertion of Proposition~\ref{prop:Weyl_elliptic}\,(iv) does not extend
to the case $\mu = \min \sigma(A_{\rm N})$.
In fact, if $\min \sigma(A_{\rm N})$ is an eigenvalue of $A_{\rm N}$,
then $\min \sigma(A_{\rm N})$ is a (generalized) pole of order one
of the analytic (Nevanlinna) function $\lambda \mapsto \overline{M (\lambda)}$ and thus
\begin{align*}
  \big\| \overline{M(\lambda)}\big\| \sim \frac{K}{\min\sigma(A_{\rm N})-\lambda}
  \qquad \text{as} \quad \lambda\nearrow\min\sigma(A_{\rm N})
\end{align*}
for some $K > 0$.

Nevertheless, in some cases the estimate~\eqref{estimate} holds
even for $\mu = \min \sigma(A_{\rm N})$.  For instance,
in the case of the Laplacian $\cL = -\Delta$ on the half-space
\begin{align*}
  \Omega = \dR^n_+ = \left\{(x',x_n)^\top\in\dR^n: x'\in\dR^{n-1},\,x_n > 0\right\}
\end{align*}
with boundary $\partial\Omega = \dR^{n-1}$ one has $\sigma(A_{\rm N}) = [0,\infty)$,
and the Neumann-to-Dirichlet map can be calculated explicitly, namely,
\begin{equation*}
  \ov{M(\lambda)} = (-\Delta_{\dR^{n-1}}-\lambda)^{-1/2}, \qquad \lambda < 0,
\end{equation*}
where $-\Delta_{\dR^{n-1}}$ is the self-adjoint Laplacian in $L^2(\dR^{n-1})$;
see, e.g.\ \cite[(9.65)]{G09}. % or~\cite{FL08,LR12}.
From this representation it follows that
\[
  \big\|\ov{M(\lambda)}\big\| = \frac{1}{\sqrt{-\lambda}\,}\,, \qquad \lambda < 0.
\]
We remark that the asymptotic behaviour of the Neumann-to-Dirichlet map
for the Laplacian on a half-space was also used in the context of spectral theory in,
e.g.~\cite{LR12}.
\end{remark}

The following theorem shows the self-adjointness of elliptic differential operators
with generalized Robin boundary conditions and yields a bound for the
minima of their spectra. Note that the $\gamma$-field corresponding to the
quasi boundary triple in Proposition~\ref{prop:QBTelliptic} is the
mapping $\varphi\mapsto \gamma(\lambda)\varphi=f$, where $f\in\dom T$
is the unique solution of the boundary value problem $\cL f=\lambda f$,
$\Gamma_0 f=\frac{\partial f}{\partial\nu_\cL}|_{\partial\Omega}=\varphi\in H^{1/2}(\partial\Omega)$.

\begin{theorem}\label{thm:mainApplication}
Let $B$ be a self-adjoint operator in $L^2 (\partial \Omega)$ which is bounded
from above and assume that $H^{1/2}(\partial\Omega) \subset \dom B$
and $B(H^1(\partial\Omega)) \subset H^{1/2}(\partial\Omega)$.  Then the operator
\begin{equation*}
  A_{[B]}f = \cL f, \qquad
  \dom A_{[B]} = \left\{f \in H^2(\Omega): \frac{\partial f}{\partial\nu_\cL}
  \Big|_{\partial \Omega} = B f |_{\partial \Omega}\right\},
\end{equation*}
is self-adjoint in $L^2 (\Omega)$ and
\begin{equation*}
  (A_{[B]}-\lambda)^{-1}=(A_{\rm N}-\lambda)^{-1}
  +\gamma(\lambda)\bigl(I-BM(\lambda)\bigr)^{-1}B\gamma(\overline\lambda)^*
\end{equation*}
holds for all $\lambda\in\rho(A_{[B]})\cap \rho(A_{\rm N})$, where $\gamma$ is
the $\gamma$-field corresponding to the quasi boundary triple
in Proposition~\ref{prop:QBTelliptic} and $M$ is the Neumann-to-Dirichlet map.
Moreover, the self-adjoint operator $A_{[B]}$
is bounded from below  with lower bound satisfying
\begin{equation*}
  \min\sigma(A_{[B]}) \ge \mu-(C\|B_+\|)^2
\end{equation*}
for each $\mu < \min\sigma(A_{\rm N})$, where $C = C (\Omega, \cL, \mu)$
is given in Remark~\ref{rem:Konstante} and $B_+$ is the positive part of~$B$;
see \eqref{eq:negPos}.  Moreover, if $B\le0$
then $\min \sigma(A_{[B]}) \geq \min \sigma(A_{\rm N})$.
\end{theorem}

\begin{proof}
Propositions~\ref{prop:QBTelliptic} and \ref{prop:Weyl_elliptic} show that
all assumptions of Theorem~\ref{thm:lowerBound} are satisfied.
Hence the latter yields all assertions of the current theorem.
\end{proof}

\begin{remark}
The boundary conditions discussed in Theorem~\ref{thm:mainApplication} contain
classical Robin boundary conditions.
Here one chooses $B f = b f$, $f \in L^2 (\partial \Omega)$, for some
suitable function $b:\partial\Omega\to\dR$ satisfying the
assumptions in the theorem; in this case
\begin{align*}
  B_+ f = b_+ f, \qquad f \in \dom B,
\end{align*}
where $b_+$ is the positive part of the function $b$.
Note also that not every Robin boundary condition with a real valued function $b$  leads to a
self-adjoint realization.  An example with an unbounded $b$ and its physical motivation
were  discussed in \cite[Section~3]{ES88}; see also
\cite{B09,MR09} for related problems.
\end{remark}

Let us formulate a consequence of the previous theorem:
under the assumptions of Theorem~\ref{thm:mainApplication} the operators
\begin{align*}
  A_{[\omega B]} f = \cL f , \quad
  \dom A_{[\omega B]} = \left\{ f \in H^2 (\Omega) :
  \frac{\partial f}{\partial \nu_\cL} \big|_{\partial \Omega}
  = \omega B f |_{\partial \Omega} \right\},
\end{align*}
with $\omega \geq 0$ are self-adjoint in $L^2 (\Omega)$ and bounded from below,
and as in Corollary~\ref{cor:Asymptotics}\,(a) one can derive the
following asymptotic estimate for the lower bound of $\sigma(A_{[\omega B]})$
when the coupling constant $\omega$ tends to $+ \infty$.

\begin{corollary}
Let the assumptions of Theorem~\ref{thm:mainApplication} be satisfied.  Then
\begin{equation*}
  \bigl|\min\sigma(A_{[\omega B]})\bigr| = \cO(|\omega|^2) \quad \text{as} \quad
  \omega \to +\infty.
\end{equation*}
\end{corollary}

\begin{remark}\label{formremark}
We point out that the operator $A_{[B]}$ in Theorem~\ref{thm:mainApplication}
can also be defined as the self-adjoint operator representing
the closed, densely defined, symmetric and lower-semibounded sesquilinear form
\begin{equation*}
  \fra_B [f,g] \defeq
  \fra[f,g]-\bigl(Bf|_{\partial\Omega},g|_{\partial\Omega}\bigr),
  \qquad \dom\fra_B \defeq H^1(\Omega),
\end{equation*}
where $\fra$ is defined as in \eqref{defa} and~$(\cdot, \cdot)$ denotes
the inner product in~$L^2 (\partial \Omega)$.
With this approach additional arguments are needed to
show $H^2$-regularity of the functions in $\dom A_{[B]}$.
Note that the decomposition $B= B_+ - B_-$ yields the estimate
\begin{equation}
\label{fraBlowerbound1}
  \mathfrak{a}_B [f]\geq \mathfrak a[f]
  -\bigl(B_+ f|_{\partial\Omega},f|_{\partial\Omega}\bigr), \qquad
  f\in H^1(\Omega).
\end{equation}
Moreover, for any $\eps>0$
there exists a constant $\beta(\eps)>0$ such that
\begin{equation*}
  \|f|_{\partial\Omega}\|^2
  \le \eps \|\nabla f\|^2 + \beta(\eps)\|f\|^2, \qquad
  f\in H^1(\Omega);
\end{equation*}
see, e.g.\ \cite[Lemma~3.1]{F67}.  According to this estimate we have
\begin{equation*}
  \bigl(B_+ f|_{\partial\Omega},f|_{\partial\Omega}\bigr)_{\partial \Omega}
  \le \|B_+\|\|f|_{\partial\Omega}\|^2\le\eps\|B_+\|\|\nabla f\|^2+ \beta(\eps)\|B_+\|\|f\|^2
\end{equation*}
for all $f\in H^1(\Omega)$.
The ellipticity of $\cL$ yields $E \|\nabla f\|^2 \leq \fra_1[f]$
for all $f\in H^1(\Omega)$, where~$E$ is chosen as in~\eqref{ellipticity}.
Thus, if $B_+ \neq 0$, then for $\eps = \tfrac{E}{\|B_+\|}$ and any $f\in H^1(\Omega)$
we obtain the relation
\begin{align*}
  \fra_B [f]
  &\ge
  E \|\nabla f\|^2 + \essinf a \|f\|^2
  - E\|\nabla f\|^2 - \beta \big(\tfrac{E}{\|B_+\|} \big)\|B_+\| \|f\|^2 \\
  &= \Bigl(\essinf a - \beta \bigl(\tfrac{E}{\|B_+\|}\bigr)\|B_+\|\Bigr)\|f\|^2,
\end{align*}
where we used \eqref{fraBlowerbound1}. In particular,
\begin{align}\label{eq:alternativeEstimate}
  \min\sigma(A_{[B]}) \ge \essinf a - \beta \bigl(\tfrac{E}{\|B_+\|}\bigr)\|B_+\|.
\end{align}
The possible choice of the constant $\beta (\eps)$ depends on the
domain $\Omega$ and has been investigated for certain classes of domains.
For instance, if $\Omega$ is a bounded domain (with Lipschitz boundary),
then one can choose $\beta$ such that $\beta(\eps) = c/\eps$ for
sufficiently small $\eps$ and some $c > 0$;
see \cite[Lemma~2.5]{GM08}. In this case \eqref{eq:alternativeEstimate} reads
\begin{equation*}
  \min\sigma(A_{[B]}) \ge \essinf a - \frac{c}{E} \|B_+\|^2
\end{equation*}
when $\|B_+\|$ is sufficiently large.
\end{remark}

% *******************************************************************

\end{document}